\documentclass[a4paper,12pt]{article}
\usepackage[english]{babel} 
\usepackage{amsfonts,amssymb,amsmath,a4wide,amsbsy} 
\usepackage{amsthm} 
\usepackage{cases}
\newcommand{\thefont}[2]{\fontsize{#1}{#2}\fontshape{n}\selectfont}
\def\ind{\rlap{\thefont{10pt}{12pt}1}\kern.16em\rlap{\thefont{11pt}{13.2pt}1}\kern.4em}

\providecommand{\be}{\begin{equation}}
\providecommand{\ee}{\end{equation}}

\theoremstyle{plain}
\newtheorem{thm}{Theorem}
\newtheorem{lem}[thm]{Lemma}
\newtheorem{cor}[thm]{Corollary}

\theoremstyle{definition}

\title{An improvement of the Boppana-Holzman bound for Rademacher random variables}
\author{Harrie Hendriks and Martien C.A. van Zuijlen}
\date{}

\begin{document}

\maketitle
\centerline{DEPARTMENT OF MATHEMATICS} 
\medskip
\centerline{RADBOUD UNIVERSITY  NIJMEGEN}
\begin{center}
	{Heyendaalseweg 135\\ 6525 AJ Nijmegen \\The Netherlands\\
	e-mail corresponding author: M.vanZuijlen@science.ru.nl\\
	\textbf{keywords} Randomly signed sums, Tomaszewski's problem}
\end{center}

\begin{abstract}
Let $v_1,v_2,...,v_n$  be real numbers whose squares add up to $1$. 
Consider the $2^n$ signed sums of the form $S=\sum_{i=1}^n \pm v_i.$ 
Holzman and Kleitman (1992) proved that at least 
$\frac38=0.375$ of these sums satisfy $|S|\leq 1.$ 
By  using bounds for appropriate moments of $S,$ 
Boppana and Holzman (2017) were able to improve the bound to 
$\frac{13}{32}=0.40625$ and even a bit better to $\frac{13}{32}+9\times10^{-6}.$ 
By following their approach, but using a key result of Bentkus and Dzindzalieta (2015),
we will drastically improve (by more than 5\%) the latter barrier $\frac{13}{32}$ to
$\frac{1}{2}-\frac{\Phi(-2)}{4\Phi(-\sqrt{2})}\approx  0.42768.$
\end{abstract}

\noindent
\textbf{Notice:}
At this place we would like to mention that since the submission of our original manuscript
Ravi Boppana joined with the authors to produce a succinct account of the improvement
in  Boppana, R., Hendriks, H. and Van Zuijlen, M.C.A. (2020), \emph{Tomaszewski's problem on randomly signed sums, revisited},
arXiv:2003.06433 [math.CO].
Moreover Nathan Keller and Ohad Klein submitted a manuscript of a complete proof of the conjecture:
Keller, N. and Klein, O. (2020), \emph{Proof of Tomaszewski's Conjecture on Randomly Signed Sums},
 	arXiv:2006.16834 [math.CO].

 \section{Introduction and  main result}

In this note we will present a considerable improvement on a result of Boppana and Holzman (2017).
We will combine their approach in \cite[Theorem 4]{BH}, based on stopping times, which is a technique
initiated by Ben-Tal \emph{et. al.} \cite{B-T} and refined by Shnurnikov \cite{S}, with a useful result for sums of Rademacher random variables of Bentkus and Dzindzalieta \cite{BD}.

Throughout this paper $n$ is a positive integer, $\epsilon_1,\epsilon_2,...,\epsilon_n$ are iid Rademacher random variables,  $\Phi$ is the standard normal distribution function and the decreasing functions $G$ and $F$ on $(0,\infty)$ are defined as follows
$$G(c)=\frac12\left(1-\frac12\,\frac{1-\Phi(c^{-1/2})}{1-\Phi(\sqrt2)}\right);\;\;\;
F(c)=\frac{1}{2}(1-3c^2).$$
\\
The main result in this paper is the following theorem.   
   \begin{thm}\label{Stelling}
   Let $v_1,v_2,...,v_n$ be real numbers such that
  $\sum_{i=1}^n v_i^2\leq 1$ and let 
  $S:=\sum_{i=1}^n v_i\epsilon_i.$ 
  Then, $$ P(|S|\leq 1)\geq G(\frac{1}{4})=\frac{1}{2}-\frac{1-\Phi(2)}{4(1-\Phi(\sqrt{2}))}\approx 0.42768 .$$
  \end{thm} 
%
Notice that Boppana and Holzman \cite{BH} found the lower bound $F(\frac14)=13/32=0.40625$,
and additionally improved this result by the term $9\times10^{-6}$. 
For $n\leq 9$ the optimal lower bound  $\frac{1}{2}$  has been obtained in Hendriks and Van Zuijlen \cite{HvZ}.  Also, Van Zuijlen \cite{vZ} obtained this lower bound $\frac{1}{2}$ in case $v_1=v_2=...=v_n,$ and thus solved the old conjecture
of B. Tomaszewski (1986) (see Guy \cite{RG})  in the uniform case.

 Notice that
$G(0):=\lim_{c\downarrow0}G(c)=\frac12$ and
 $F$ is already used in \cite{BH}. 
In the Appendix we prove that $G(c)> F(c)$ for all $c>0,$ but
the important fact is that $G(\frac14)>F(\frac14)$.

For the proof of Theorem  \ref{Stelling} we will need the following improvement of Lemma 3 in \cite{BH}.

  \begin{lem}\label{Ons Lemma 3}
  Let $x$ be a real number such that $|x|\leq 1,$ let $v_1,v_2,...,v_n$ be real numbers such that
\begin{equation}\label{U}
\sigma^2:=\sum_{i=1}^n v_i^2\leq U(1+|x|)^2,
\end{equation}
 and let $Y:=\sum_{i=1}^n v_i\epsilon_i$. 
 Then,
 $$  P(|x+Y|\le 1)\geq 
  \frac{1}{2}-\frac{1}{4}\frac{1-\Phi (U^{-1/2})}{1-\Phi (\sqrt 2)}
  =G(U).$$
  Notice that $G(U)\uparrow \frac{1}{2},$ as $U \downarrow 0.$
  \end{lem}

  \begin{proof}
  Because of symmetry around zero of the distribution of $Y$ we may assume that $0\le x\le 1,$ and then
   \begin{align*}
   P(|x+Y| \le 1)&=P(-1-x \le Y \le 1-x)=\frac{1}{2}P(|Y| \le 1+x)+\frac{1}{2}P(|Y| \le 1-x)\geq 
\\&
\geq\frac{1}{2}P(|Y| \le 1+x)=\frac{1}{2}-\frac{1}{2}P(|Y| > 1+x )=\frac{1}{2}-P(Y > 1+x ).
  \end{align*}
   Moreover,  from Bentkus and Dzindzalieta \cite{BD}, or Dzindzalieta's thesis \cite{D} p. 30, Theorem 11, we have for $x\in [0,1],$  with $c_*:=\frac{1}{4(1-\Phi(\sqrt2)}=3.178...,$
  $$P(Y\geq 1+x)\leq P\{Y\geq U^{-1/2}\sigma\}=P(\frac{Y}{\sigma}\geq  U^{-1/2})   \leq c_*(1-\Phi(U^{-1/2}))=\frac{1-\Phi(U^{-\frac{1}{2}})}{4(1-\Phi(\sqrt2)},$$
so  that 
$$P(|x+Y| \le 1)\geq \frac{1}{2}-P(Y > 1+x)\geq 
\frac{1}{2}-\frac{1}{4}\frac{1-\Phi (U^{-1/2})}{1-\Phi (\sqrt 2)}=G(U).$$
   \end{proof}
In \cite{BH} a first improvement of the lower bound 3/8 is based on condition (\ref{U}) 
with $U=2/7$, their further improvement Theorem 4 is roughly based on condition (\ref{U}) 
with $U=1/4$.

   \begin{cor}
By taking $U=\frac{2}{7}$ in Lemma \ref{Ons Lemma 3}, we obtain
$$\left[\sum_{i=1}^n v_i^2\leq \frac{2}{7}(1+|x|)^2\right]\Rightarrow 
  \left[ P(|x+Y| \le 1)\geq G(\frac{2}{7})= \frac{1}{2}-\frac{\Phi(-\sqrt{7/2})}{4\Phi(-\sqrt{2})}\approx 0.40246\right],$$
  and by taking $U=\frac{1}{4},$ we obtain
 $$ \left[\sum_{i=1}^n v_i^2\leq \frac{1}{4}(1+|x|)^2\right]\Rightarrow 
  \left[P(|x+Y| \le 1)\geq G(\frac{1}{4})= \frac{1}{2}-\frac{\Phi(-2)}{4\Phi(-\sqrt{2})}\approx  0.42768\right].$$
\end{cor} 
~\\

  \section{Proof of the Theorem}
The proof of Theorem 4 in Boppana and Holzman can be followed with the  exception  that the foregoing Lemma \ref{Ons Lemma 3} is used instead of their Lemma 3. Lemma \ref{Ons Lemma 3} is based on a crucial inequality of  Bentkus and Dzindzalieta \cite{BD}. 
To indicate precisely where the differences occur, we will present the complete proof.
  
   Assume $\sum_{i=1}^nv_i^2\leq 1.$ By inserting zeroes, we may assume that $n\geq 4$ and without loss of generality, by reordering the real numbers, we assume
$$v_n \geq v_{1}  \geq v_{n-1}\geq v_{2}\geq v_{3}\geq ...\geq v_{n-2}\geq 0,$$
so that  
$$v_n+ v_1+ v_{n-1}+ v_{2}\leq \sqrt 4\times\sqrt {v_n^2+ v_1^2+ v_{n-1}^2+ v_{2}^2} \leq 2\sqrt{\sum_{i=1}^n}v_i^2\leq 2,$$ and hence $v_1+v_2\leq 1.$
\\

{\bf STOPPING TIMES:}
For $t\in \{1,2,...,n-1\},$ define 
$$X_t:=\sum_{i=1}^t v_i\epsilon_i,\quad Y_t:=\sum_{i=t+1}^n v_i\epsilon_i,$$
so that $S=X_t+Y_t=\sum_{i=1}^nv_i\epsilon_i$.
Moreover, let
$$A:=\{t\mid t\leq n-1\wedge |X_t|> 1-v_{t+1}\}\subset\{1,\ldots,n-1\};\;\;\;\;
T:=\min(A\cup\{n-1\}).$$
Hence, if $T\le n-2$, then $T$ is the first time the process 
$|X_t|$ exceeds the boundary $1-v_{t+1}$ and
 $T=n-1$ iff $|X_t|\le 1-v_{t+1}$ for all $t\le n-2$.
Since $v_1+v_2\leq 1$ and
in particular,  $|X_1|=v_1\leq 1-v_2,$ we have  $T\geq 2.$ 
Also,
$$
[|X_{s}|\leq 1-v_{s+1},\forall s< T];\quad
|X_{T}| \leq 1;\quad
[[T\le n-2]\Rightarrow [|X_{T}|>1-v_{T+1}]].
$$

Similarly, for $t\in\{1,...,n\},$ define $M_t:=\sum_{i=1}^tv_i$ and let 
$$B:=\{t\mid t\leq n-1 \wedge \;M_t> 1-v_{t+1}\}\subset\{1,2,\ldots,n-1\};\;\;\;\;
K:=\min (B\cup\{n-1\}).$$
In fact, if $K\le n-2$, then $K$ is the first time the process $M_t$ exceeds the boundary $1-v_{t+1}$ and
$K=n-1$ iff $M_t\le 1-v_{t+1}$ for all $t\le n-1$.
Notice that $2\leq K\leq T,$ in contrast to $T,$ $K$ is not random and
$$
[M_{s}\leq 1-v_{s+1},\forall s< K];\quad
M_{K} \leq 1< M_{K+1}  ;\quad
[[K\le n-2]\Rightarrow [M_{K}>1-v_{K+1}]].
$$

  To prove our Theorem \ref{Stelling} we  may assume by symmetry that $X_T\geq 0.$ 
We will divide the proof  into some cases, depending on $T.$
   
  First of all we remark that
  for $i\in\{n-1,n-2\}$ 
$$
P(|S|\leq 1\mid T=i,X_T)
\geq P(Y_T\leq 0\mid T=i,X_T)\geq \frac{1}{2}\geq G(\frac{1}{4})
$$
 and hence also
 \begin{equation}\label{Case12}
P(|S|\leq 1\mid T=i)\geq\frac12\geq G(\frac{1}{4}),\text{ for }i\in\{n-1,n-2\},
\end{equation}
since for
 $T=n-1,$ we have $0\leq |X_T|=X_T\leq 1$ and $|Y_T|=v_n\leq 1,$ so that
$$[Y_T\leq 0]\Rightarrow [-1\leq X_T+Y_T=S\leq 1],$$
whereas
for $T=n-2,$  we have $1-v_{n-1}< X_T\leq 1$ and $|Y_T|\leq v_{n-1}+v_n\leq \sqrt 3-v_1,$ so that $v_{n-1}-\sqrt 3\leq v_1-\sqrt3\leq Y_T,$ and hence  $$[ Y_T\leq 0]\Rightarrow [-1\leq 1-v_{n-1}+v_{n-1}-\sqrt 3\leq X_T+Y_T=S\leq 1].$$

Next, we claim that with
$$U_K(i):=\frac{(K+1)^2-i}{(2K+1)^2}$$ we have
\begin{numcases}
{\sum_{i=T+1}^{n}v_i^2\leq }
U_K(T)(1+X_T)^2,
\text{ for }\;\; 2\le K\le T\le\frac{3K+2}{2},~T\le n-3 \label{Case3}
\\ 									
 U_K\left(\frac{3K+2}{2}\right)(1+X_T)^2,\text{ for } \;\;\frac{3K+2}{2}\le T\le n-3. \label{Case4}
\end{numcases}
To show (\ref{Case3}), let
$2\leq K\leq T\leq \frac{3K+2}{2}$ and $T\leq n-3.$
Clearly,  we have (Cauchy-Schwartz) for $K=1,2,...,n-2,$
 (hence $M_{K+1}>1$),
$$1\geq \sum_{i=1}^{K+1}v_i^2\geq \frac{1}{K+1}M_{K+1}^2> \frac{1}{K+1}.$$ 
Therefore, since $v_{T+1}>1-X_T$ 
for  $T\leq n-2,$ we obtain for $3\leq K+1\leq T\leq n-3,$
$$\sum_{i=1}^{T}v_i^2> B_1:=\frac{1}{K+1}+(T-K-1)(1-X_T)^2;\quad\sum_{i=1}^{T}v_i^2>B_2:=T(1-X_T)^2.$$
For $T=K\le n-3$ we stil have $\sum_{i=1}^Tv_i^2>B_2$ and
$$\sum_{i=1}^{T}v_i^2\geq \frac{1}{T}(\sum_{i=1}^{T}v_i)^2= \frac{1}{T}M_T^2\geq \frac{1}{T}X_T^2=\frac{1}{K}X_T^2\geq	\frac{1}{K+1} -(1-X_T)^2,$$
where the last inequality is strict if and only if $X_T\ne K/(K+1)$.
Notice that $$[B_1\geq B_2]\Leftrightarrow [1-X_T\leq \frac{1}{K+1}]\Leftrightarrow [X_T\in [\frac{K}{K+1},1]]\Leftrightarrow [K\leq \frac{X_T}{1-X_T}=K_0=K_0(X_T)],$$
 i.e. for "small" $K$ we have $B_1\geq B_2$ and for  "large" $K$ we have $B_1\leq B_2.$
 
\noindent It follows that, for $2\leq K\leq T\leq \frac{3K+2}{2}$ and $ T\leq n-3,$ we have with $\lambda=\frac{2T-K-1}{2K+1}\geq 0$ and $1-\lambda=\frac{3K+2-2T}{2K+1},$ 
$$\sum_{i=1}^{T}v_i^2= \lambda\sum_{i=1}^{T}v_i^2+(1-\lambda)\sum_{i=1}^{T}v_i^2\geq \lambda B_1+(1-\lambda)B_2=$$ $$=\frac{2T-K-1}{(K+1)(2K+1)}+ \frac{(K+1)^2-T}{2K+1}(1-X_T)^2:=B,$$
so that
 $$\min(B_1,B_2)\leq B\leq \max(B_1,B_2),$$ and (as in Boppana and Holzman (2017))
\begin{align*}
\sum_{i=T+1}^{n}v_i^2
&\leq 
1-\max(B_1,B_2)\leq 1-B= \frac{(K+1)^2-T}{
(K+1)(2K+1)}[2-(K+1)(1-X_T)^2)]\leq
\\&\leq 
\frac{(K+1)^2-T}{(2K+1)^2}(1+X_T)^2=U_K(T)(1+X_T)^2.  
\end{align*}
Notice that $U_K(i)\le\frac14$ if and only if $i\ge K+\frac34$, and that $U_K(K)>\frac14$.

To show (\ref{Case4}), let  $ \frac{3K+2}{2}\leq T\leq n-3.$ Then, the upper bound above is still valid, since in this case we have
$$\sum_{i=1}^T v_i^2>\frac{1}{K+1}+(T-K-1)(1-X_T)^2,$$
so that for $T>\frac{3K+2}{2}$ we have 
$$\sum_{i=1}^T v_i^2>\frac{1}{K+1}+\frac{K}{2}(1-X_T)^2,$$
 which is exactly bound B given in 
Equation (\ref{Case3}) evaluated at $T=\frac{3K+2}{2},$ (where $\lambda=1,B=B_1),$
 so that we obtained (\ref{Case4}).

 Summarizing, we obtained for $ T\leq n-3$
the inequalities (\ref{Case3}) and (\ref{Case4}), %
so that it follows from Lemma \ref{Ons Lemma 3} by taking $x=X_T$ and $Y=Y_T,$ that for $i\leq n-3$ we have
\begin{equation*}
 P(|S|\le 1\mid T=i,X_T)\geq  
 \begin{cases}  
 G(U_K(i)),&\text{for }K\leq i\leq \frac{3K+2}{2}\\
 G(U_K(\frac{3K+2}2)),& \text{for } \frac{3K+2}{2} .
 \end{cases}
 \end{equation*} 
 and hence also
  \begin{equation}\label{PCase34}
 P(|S|\le1\mid T=i)\geq  
 \begin{cases} 
 G(U_K(i)),&\text{for }K\leq i\leq \frac{3K+2}{2}\\
 G(U_K(\frac{3K+2}2)),& \text{for } \frac{3K+2}{2}.
 \end{cases}
 \end{equation} 
 
We can now finish the proof. 
We have to deal with the problem that $U_K(K)>\frac14$.
As in Boppana and Holzmann, \cite[p. 8]{BH}, we remark that in case $K\leq n-4,$ we have $T=K$ if the signs of $\epsilon_1,\epsilon_2,...,\epsilon_K$ are all equal (probability $1/2^{K-1}$) and otherwise $T\geq K+2.$ Namely, if $\epsilon_1,\epsilon_2,...,\epsilon_K$ are not all equal, then
$|X_K|\leq 1-v_{K+1}$ and $|X_{K+1}|\leq 1-v_{K+2},$ so that  $T \geq K+2,$ since by the ordering of the $\nu_i,$
$$|X_K|\leq \sum_{i=1}^{K-1}v_i-v_K=
 M_{K-1}-v_K \leq
 1-v_K-v_K \leq 1-v_{K+1}$$ and similarly also (notice that $K\neq n-3$)
$$|X_{K+1}|\leq |X_{K}| +v_{K+1}\leq 1-2v_K+v_{K+1}\leq 1-v_K\leq 1-v_{K+2}.$$
Therefore, it follows  from (\ref{PCase34}), the fact that these bounds are non-decreasing in $T,$  the inequality $K+2\leq \frac{3K+2}{2}$ and Lemma\;\ref{App} in the Appendix that for $K\leq n-4,$
\begin{align*}
P(|S|&\le  1\mid T\le n-3)=
\\&=\frac{1}{2^{K-1}}P(|S|\leq 1\mid K=T\leq n-3)
+(1-\frac{1}{2^{K-1}})P(|S|\leq 1\mid K+2\le T\le n-3)\geq 
\\&\geq 
\frac{1}{2^{K-1}}G(U_K(K))+(1-\frac{1}{2^{K-1}})G(U_K(K+2))=
\\&=  
\frac{1}{2^{K-1}}G\left(\frac{K^2+K+1}{(2K+1)^2}\right)  +  (1-\frac{1}{2^{K-1}})G\left(\frac{K^2+K-1}{(2K+1)^2}\right)\ge G(\frac14).
\end{align*}
Hence, in the situation  $K\le n-4,$ we obtain the lower bound 
$$P(|S|\leq 1\mid T\le n-3)\ge G(\frac14)\approx 
0.427685.$$

Finally, as in Boppana and Holzmann, \cite{BH}, one can get rid of the restriction $K\leq n-4. $ Namely, for $K=n-3$ it is still true that $P\{T=K\}=\frac{1}{2^{K-1}}$, and while $T=K+1=n-2$ may occur in this case, it yields a conditional bound of $\frac{1}{2}$ as given in (\ref{Case12}) above. 
Hence, from (\ref{Case12}) and (\ref{PCase34}) we obtain in case
$K=n-3,$ 
$$P(|S|\leq 1\mid K\le T)=\frac{1}{2^{K-1}}P\{|S|\leq 1\mid K=T\}+(1-\frac{1}{2^{K-1}})P\{|S|\leq 1| K+1\leq T\}\geq $$
$$\geq \frac{1}{2^{K-1}}G(U_K(K))+(1-\frac{1}{2^{K-1}})\times\frac{1}{2} \ge G(\frac14),$$
as shown in Lemma \ref{App} in the Appendix.
%

The cases $K=n-2$ and $K=n-1$ (hence $T\geq n-2),$  are covered by the conditional bound of $\frac{1}{2}$ in (\ref{Case12})
.

\section{Appendix}

\begin{lem}\label{App}
For all $x>0$ we have
$$G(x):=\frac12\left(1-\frac12\,\frac{1-\Phi(x^{-1/2})}{1-\Phi(\sqrt2)}\right)>\frac{1}{2}(1-3x^2):=F(x).$$
Moreover, with $p_k:=2^{1-k},$ we have for $k\geq 2,$
$$ h(k):=p_kG(\frac{k^2+k+1}{(2k+1)^2})  +  (1-p_k)
G(\frac{k^2+k-1}{(2k+1)^2})\geq G(\frac{1}{4}).$$
Since $G$ is decreasing and $G(0)=\frac12$, we also have for $k\geq 2,$
$$ p_kG(\frac{k^2+k+1}{(2k+1)^2})  +  (1-p_k)
\frac12\geq G(\frac{1}{4}).$$

 \end{lem}
\begin{proof}For the first statement we have to show that 
\begin{equation*}
\left[\bar{\Phi}(x^{-1/2})<\frac{3}{8\bar{\Phi}(\sqrt2)}x^2=\frac{3}{2c_*}x^2\right]\hbox{ or equivalently } \left[\frac{\bar{\Phi}(x^{-1/2})}{x^2}<\frac{3}{8\bar{\Phi}(\sqrt2)}\approx 0.4714
\right],\end{equation*}
where $\bar\Phi(x)=1-\Phi(x)=\Phi(-x)$, and as in Bentkus and Dzindzalieta \cite{BD} 
$$c_*:=\frac{1}{4\bar{\Phi}(\sqrt{2})}\approx 3.178
 .$$
By substituting $y=\frac{1}{\sqrt{x}},$ it is  equivalent with showing that for $y> 0,$
 $$H(y):=y^4\bar{\Phi}(y)\leq \frac{3}{8\bar{\Phi}(\sqrt2)}\approx 0.4714.$$
However,
   since for $y>0,$ we have $y\bar{\Phi}(y)=y\int_y^\infty \phi(x)dx\leq \int_y^\infty x\phi(x)dx=\phi(y),$
    it is sufficient to show that for $y>0,$  
$$L(y):=y^3\phi(y)\leq 0,4714.$$ 
The function $L$ has a maximum in $y=\sqrt3$ and $L(\sqrt3)\approx 0,4625<0,4714,$ so that we are done.
\\
\\
The second inequality in the Lemma is equivalent to
$$
p_k\Phi(b^{-1/2})+(1-p_k)\Phi(a^{-1/2})\ge\Phi(2),
$$
where
$$0\le a:=\frac{k^2+k-1}{(2k+1)^2}<\frac14<b:=\frac{k^2+k+1}{(2k+1)^2}\le\frac13.
$$
It is sufficient to prove this inequality with $p_k$ replaced by $p_2=\frac12$,
since $\Phi$ is increasing and $p_k\le p_2=\frac12$ so that
$$
p_k\Phi(b^{-1/2})+(1-p_k)\Phi(a^{-1/2})
\ge \frac12\Phi(b^{-1/2})+\frac12\Phi(a^{-1/2}).
$$
Consider
$$Z_\xi(\varepsilon)=\frac{2} {\sqrt{1+\xi\varepsilon}}.$$
Notice that with $\varepsilon=(k+1/2)^{-2}$ we have $a^{-1/2}=Z_\xi(\varepsilon)$ for $\xi=-5/4$ and $b^{-1/2}=Z_\xi(\varepsilon)$ for $\xi=3/4$.
Denote the density function of the standard normal distribution by $\varphi$, so that the derivative $\Phi'$ statisfies $\Phi'=\varphi$.
Consider the composition $(\Phi Z_\xi)(\varepsilon)=\Phi(Z_\xi(\varepsilon))$, then using
$\varphi'(z)=-z\varphi(z)$ one finds 
$$(\Phi Z_\xi)''(\varepsilon)=-\frac12\cdot(\varphi Z_\xi)(\varepsilon)(1+\xi\varepsilon)^{-7/2}\xi^2(1-3\xi\varepsilon).$$
We conclude that $(\Phi Z_\xi)(\varepsilon)$ is concave function in $\varepsilon$ if $(1-3\xi\varepsilon)\ge0$.
Thus for $\xi=3/4$ we need $\varepsilon\le4/9=(1+1/2)^{-2}$. Hence $(\Phi Z_\xi)(\varepsilon)$ is concave on $[0,4/9].$
It is clear that we have $\Phi(Z_\xi(0))=\Phi(2)$ and for $k=1$ we have $\varepsilon=4/9$, $a=1/9$ and $b=1/3,$ so that
$$\frac12\Phi(b^{-1/2})+\frac12\Phi(a^{-1/2})=\frac12\Phi(\sqrt3)+\frac12\Phi(3)\ge0.9785>0.9773>\Phi(2).
$$
This proves the second statement of the Lemma.
 \end{proof}

\end{document}